\documentclass[12pt]{amsart}
\usepackage{amstext}
\usepackage{amsthm}
\usepackage{amsmath}
\usepackage{amssymb}
\usepackage{latexsym}
\usepackage{amsfonts}
\usepackage{graphicx}

\usepackage[pagebackref,hypertexnames=false, colorlinks, citecolor=red, linkcolor=red]{hyperref} 
\usepackage[backrefs]{amsrefs}

\allowdisplaybreaks

\DeclareMathOperator{\trace}{trace}

\newcommand{\norm}[1]{\ensuremath{\left\|#1\right\|}}
\newcommand{\abs}[1]{\ensuremath{\left\vert#1\right\vert}}

\newcommand{\inp}[2]{\left\langle #1, #2 \right\rangle}
\newcommand{\mult}{\mathrm{mult}}

\newcommand{\D}{\mathbb{D}}

\newcommand{\R}{\mathbb{R}}

\newcommand{\C}{\mathbb{C}}
\newcommand{\N}{\mathbb{N}}

{\end{list}}

\numberwithin{equation}{section}

\newtheorem{thm}{Theorem}[section]
\newtheorem{lm}[thm]{Lemma}
\newtheorem{cor}[thm]{Corollary}

\newtheorem{prop}[thm]{Proposition}
\newtheorem*{prop*}{Proposition}

\theoremstyle{remark}

\newtheorem*{rem*}{Remark}

\title[Remarks about Interpolating Sequences]{Some Remarks about Interpolating Sequences in Reproducing Kernel Hilbert Spaces}

\author[M. Raghupathi]{Mrinal Raghupathi}
\address{Mrinal Raghupathi, Department of Mathematics\\ United States Naval Academy \\Annapolis, MD, 21403}
\email{raghupat@usna.edu}

\author[B. D. Wick]{Brett D. Wick}
\address{Brett D. Wick, School of Mathematics\\ Georgia Institute of Technology\\ 686 Cherry Street\\ Atlanta, GA USA 30332-0160}
\email{wick@math.gatech.edu}
\thanks{Research supported in part by a National Science Foundation DMS Grant \# 1001098 and \# 0955432}

\subjclass[2000]{Primary }

\begin{document}

\begin{abstract} 
  In this paper we study two separate problems on interpolation. We
  first give some new equivalences of Stout's Theorem on necessary and
  sufficient conditions for a sequence of points to be an
  interpolating sequence on a finite open Riemann surface.  We next turn our attention to the
  question of interpolation for reproducing kernel Hilbert spaces on
  the polydisc and provide a collection of equivalent statements about
  when it is possible to interpolation in the Schur-Agler class of the
  associated reproducing kernel Hilbert space.
\end{abstract}

\maketitle

\section{Introduction and Statement of Main Results}

Recall that a sequence $Z=\{z_j\}\subset\D$ is called an $H^\infty$-\textit{interpolating} sequence if for every $a=\{a_j\}\in\ell^\infty$ there exists a function $f\in H^\infty$ such that
$$
f(z_j)=a_j\quad\forall j.
$$
Similarly, for the sequence $Z$ let $\ell^2(\mu_Z)$ be the space of all sequences $a=\{a_j\}$ such that
$$
\sum_{j=1}^\infty\abs{a_j}^2(1-\abs{z_j}^2):=\norm{a}_{\ell^2(\mu_Z)}^2<\infty.
$$

Then the  sequence $Z=\{z_j\}$ is called an $H^2$-\textit{interpolating} sequence if for every $a=\{a_j\}\in\ell^2(\mu_Z)$ there exists a function $f\in H^2$ such that
$$
f(z_j)=a_j\quad\forall j.
$$
As is well known, these sequences turn out to be one in the same and
are characterized by a separation condition on the points in $Z$ and
that the points must generate a Carleson measure for the space
$H^2$. The following theorem gives a precise statement of this.

\begin{thm}[Carleson, \cite{Car2}, Shapiro, Shields \cite{ShSh2}]
\label{CarInterp}
The following are equivalent:
\begin{itemize}
\item[(a)] The sequence $Z$ is $H^2$-interpolating;
\item[(b)] The sequence $Z$ is $H^\infty$-interpolating;
\item[(c)] The sequence $Z$ is separated in the pseudo-hyperbolic metric and generates a $H^2$-Carleson measure.  In particular, $\sum_{z_j\in Z}(1-\abs{z_j}^2)\delta_{z_j}$ is a $H^2$ Carleson measure and
$$
\inf_{j\neq k}\abs{\frac{z_j-z_k}{1-\overline{z_k}z_j}}\geq\delta>0;
$$
\item[(d)] The sequence $Z$ is strongly separated, namely there exists a constant $\delta>0$ such that
$$
\inf_{j}\abs{\prod_{j\neq k}\frac{z_j-z_k}{1-\overline{z_k}z_j}}\geq\delta>0.
$$
\end{itemize}
\end{thm}

Since the results of Carleson, \cite{Car2}, and Shapiro-Shields, \cite{ShSh2} the question of characterizing the interpolating sequences for other spaces of analytic functions has been intensively studied.  See any of the papers \cites{MR2137874, MaSu, Seip} for various generalizations of this question.

In this paper we study the problem of interpolating sequences in two
settings. First, we consider the case of finite Riemann surfaces and
obtain a new equivalences and a different proof of a theorem of
Stout. We then go on to consider a multivariable example: the Schur
Agler class. In both cases we will make heavy use of results on Pick
interpolation.

\subsection*{Interpolation on Riemann Surfaces}
Let $\Gamma$ be a Fuchsian group acting on the unit disk. We will
assume that $\Gamma$ finitely-generated. The group $\Gamma$ acts on
$H^\infty$ by composition and the associated fixed-point algebra is
denoted $H^\infty_\Gamma$. It is known that every finite open Riemann
surface can be viewed as the quotient space of the disk by the action
of such a group. The group $\Gamma$ is finitely generated and acts
without fixed points on the disk. A major advantage of viewing the
problem in terms of fixed points is that the algebra
$H^\infty_\Gamma\subseteq H^\infty$ and this allows us to bootstrap
results about Riemann surfaces to the classical setting of the open
unit disk. There is also a reproducing kernel Hilbert space
$H^2_\Gamma$ associated with the group action. We denote by $K^\Gamma$
the reproducing kernel for the Hilbert space $H^2_\Gamma$. This is
just the set of fixed points in $H^2$. In~\cite{R} it is shown that
$H^\infty_\Gamma$ is the multiplier algebra for $H^2_\Gamma$.

Given a sequence of non-zero vectors $\{x_n\}$ in a Hilbert space $H$ we define the associated Gramian as the matrix $[\inp{x_n}{x_m}]_{m,n = 1}^\infty$. The normalized Gramian is defined as the Gramian of the sequence $\tilde{x_n}$, where $\tilde{x_n} = \frac{x_n}{\norm{x_n}}$.

Our first main result of this paper is the following theorem.
\begin{thm}\label{interpriemann}
  Let $Z = (z_n)\subseteq \mathbb{D}$ be a sequence of points in
  $H^\infty_\Gamma$ such that no two points lie on the same orbit of
  $\Gamma$, where $\Gamma$ is the group of deck transformation
  associated to a finite Riemann surface. Let $Z_n = Z \setminus
  \{z_n\}$. The following are equivalent:
  \begin{enumerate}
  \item The sequence $\{z_n\}$ is interpolating for $H^\infty_\Gamma$;
  \item The sequence $\{z_n\}$ is interpolating for $H^2_\Gamma$;
  \item The sequence $\{z_n\}$ is $H^2_\Gamma$-separated and
    $\sum_{n=1}^\infty K^\Gamma(z_i,z_i)^{-1}\delta_{z_i}$ is a
    Carleson measure;
  \item The Gramian $G =
    \left[\frac{K^\Gamma(z_i,z_j)}{\sqrt{K^\Gamma(z_i,z_i)K^\Gamma(z_j,z_j)}}\right]$
    is bounded below;
  \item There is a constant $\delta>0$ such that $\inf_{n\geq 1}
    d_{H^\infty_\Gamma}(z_n, Z_n)\geq \delta$.
  \end{enumerate}
\end{thm}

A similar result was obtained by Stout~\cite{S}. However, there are
two differences between the results obtained there and our
results. First, we use the interpolation theorem from~\cite{RW} as an
essential ingredient in our proof. This modern approach appears in the
work of Marshall and Sundberg on interpolating sequences for the
Dirichlet space. Second, our proof applies to the case of a subalgebra
of $H^\infty$ that is fixed by the action of a finitely-generated
discrete group, a more general setting than the case of a finite
Riemann surface.

\subsection*{Interpolation in the Schur-Agler Class}
We now turn to the case where the domain is $\D^d$. Here the algebra
in question is the set of functions in the Schur-Agler class.  As
motivation for our results we describe the important theorem of Agler
and McCarthy that characterizes the interpolating sequences for
$H^\infty(\D^2)$.  Recall that $H^\infty(\D^2)$ is the multiplier
algebra for the space $H^2(\D^2)$, and that this is a reproducing
kernel Hilbert space with kernel given by
\[
k_z(w)=\frac{1}{1-\overline{z_1}w_1}\frac{1}{1-\overline{z_2}w_2}
\]
for $z,w\in\D^2$.  

A sequence of points $\{\lambda_j\}\subset\D^2$ is called an
$H^\infty(\D^2)$-interpolating sequence if for any sequence of bounded
numbers $\{w_i\}$ there is a function $f\in H^\infty(\D^2)$ such that
$f(\lambda_j)=w_j$.  The sequence of points is said to be
\textit{strongly separated} if for each integer $i$ there is a
function in $\varphi_i\in H^\infty(\D^2)$ of norm at most $M$ such
that $\varphi_i(\lambda_i)=1$ and $\varphi_i(\lambda_k)=0$ for $k\neq
i$.  The result of Agler and McCarthy then gives a characterization of
the interpolating sequences for $H^\infty(\D^2)$.

\begin{thm}[Agler and McCarthy, \cite{AgMc}]
Let $\{\lambda_j\}\in\D^2$.  The following are equivalent:
\begin{itemize}
\item[(i)] $\{\lambda_j\}$ is an interpolating sequence for
  $H^\infty(\D^2)$;
\item[(ii)] The following two conditions hold
\begin{itemize}
\item[$(a)$] For all admissible kernels $k$, their normalized Gramians
  are uniformly bounded above,
$$
G^k\leq MI
$$
for some $M>0$,
item[$(b)$] For all admissible kernels $k$, their normalized Gramians
  are uniformly bounded below,
$$
G^k\geq NI
$$
for some $N>0$;
\end{itemize}
\item[(iii)] The sequence $\{\lambda_j\}$ is strongly separated and
  condition $(a)$ alone holds;
\item[(iv)] Condition $(b)$ alone holds.
\end{itemize}
\end{thm}
Here an admissible kernel is one for which the pointwise by $M_{z_j}$
is a contraction on $H(k)$, the reproducing kernel Hilbert space
on $\D^2$ with kernel $k$.

We now consider a related question, but for more general products of
reproducing kernel Hilbert spaces.  Given $k_j$ with $j=1,\ldots, d$
reproducing kernels on $\D$ with the property that
\[\frac{1}{k_j}(z,w) = 1 - \inp{b_j(z)}{b_j(w)}\]
where $b_j$ is an analytic map from $\D$ into the open unit ball of a separable
Hilbert space.

The kernel $k_j$ is the reproducing kernel
for the Hilbert space $H(k_j)$.  Let $H(k)$ denote the reproducing
kernel Hilbert space defined on $\D^d$ with reproducing kernel
$k(z,w)=\prod_{j=1}^d k_j(z_j,w_j)$ for $z,w\in\D^d$.

We define $S_{H(k)}(\D^d)$ to be the set of functions
$m:\D\to\C$ such that
$$
1-m(z)\overline{m(w)}=\sum_{j=1}^d\frac{1}{k_j}(z_j,w_j)h_j(z)\overline{h_j(w)}
$$
for functions $\{h_j\}$ defined on $\D^d$.  Note that this is the
Schur-Agler class of multipliers for $H(k)$. In the case where $k_j$ is the Szeg\"o kernel and $d = 2$ an
application of Ando's theorem shows that $H^\infty(\D^2)$ and
$S_{H(k)}(\D^2)$ coincide. In higher dimensions this is no longer the case.



Let us say that a kernel $k$ is an \textit{admissible kernel} if we have that
$$
\frac{1}{k_j}(z_j, w_j) k(z,w) = (1-\inp{b_j(z_j)}{b_j(w_j)})k(z,w)\geq
0\text{ for } j=1,\ldots, d.
$$

Given a sequence of points $\{\lambda_j\}\in\D^d$, then the normalized
Gramian of $k$ is the matrix given by
$$
G_{ij}^k=\frac{k(\lambda_i,\lambda_j)}{\sqrt{k(\lambda_i,\lambda_i)k(\lambda_j,\lambda_j)}}
$$

A sequence of points $\{\lambda_j\}\subset\D^2$ is called an
$S_{H(k)}(\D^d)$-interpolating sequence if for any sequence of bounded
numbers $\{w_i\}$ there is a function $f\in S_{H(k)}$ such that
$f(\lambda_j)=w_j$.  The sequence of points is said to be
\textit{strongly separated} if for each integer $i$ there is a
function in $\varphi_i\in S_{H(k)}(\D^d)$ of norm at most $M$ such
that $\varphi_i(\lambda_i)=1$ and $\varphi_i(\lambda_k)=0$ for $k\neq
i$.  Our second main result is the following theorem providing a
generalization of the result of Agler and McCarthy:
\begin{thm}
\label{main}
  Let $\{\lambda_j\}$ be a sequence of points in $\D^d$.  The
  following are equivalent:
\begin{itemize}
\item[(i)] $\{\lambda_j\}$ is an interpolating sequence for $S_{H(k)}(\D^d)$;
\item[(ii)] The following two conditions hold
\begin{itemize}
\item[$(a)$] For all admissible kernels $k$, their normalized Gramians
  are uniformly bounded above,
$$
G^k\leq MI
$$
for some $M>0$,
\item[$(b)$] For all admissible kernels $k$, their normalized Gramians
  are uniformly bounded below,
$$
G^k\geq NI
$$
for some $N>0$;
\end{itemize}
\item[(iii)] The sequence $\{\lambda_j\}$ is strongly separated and
   condition $(a)$ alone holds;
 \item[(iv)] Condition $(b)$ alone holds.
\end{itemize}
\end{thm}

\section{Interpolation in Riemann surfaces}
\label{s.Riemann}

Our goal in this section is to prove the analogue of Carleson's
theorem for Riemann surfaces. We view the Riemann surface as the
quotient of the disk by the action of a Fuchsian group and state our
theorems for the corresponding fixed-point algebra $H^\infty_\Gamma$.

A central result that we require is a Nevanlinna--Pick type theorem
obtained in~\cite{RW}. We briefly recall the parts of that paper that
are most relevant to our work.

Let
$C(H^\infty_\Gamma)$ be the set of columns over $H^\infty_\Gamma$,
similarly, let $R(H^\infty_\Gamma)$ denote the rows. There is a
natural identification between $C(H^\infty_\Gamma)$ and the space of
multipliers $\mult(H^2_\Gamma, H^2_\Gamma\otimes \ell^2)$.  There is
also a natural identification between $R(H^\infty_\Gamma)$ and
$\mult(H^2_\Gamma\otimes \ell^2, H^2_\Gamma)$. 
\begin{thm}[\cite{RW}]\label{interpthm}
  Let $z_1,\ldots,z_n\in\D$, $w_1,\ldots,w_n\in \C$ and
  $v_1,\ldots,v_n\in\ell^2$. There exists a function $F\in
  C(H^\infty_\Gamma)$ such that $\norm{F}\leq C$ and  $\inp{F(z_i)}{v_i} =
  w_i$ if and only if the matrix $[(\alpha^2C^2\inp{v_j}{v_i} -
  w_i\overline{w_j})K^\Gamma(z_i, z_j)]\geq 0$. The constant $\alpha$
  depends on $\Gamma$ but not on the points $z_1,\ldots,z_n$.
\end{thm} 

A similar argument also establishes the fact that there is a function
$F\in R(H^\infty_\Gamma)$ such that $\norm{F}\leq C$ and $F(z_i) =
v_i$ if and only if the matrix $[(C^2\alpha^2 -
\inp{v_j}{v_i})K^\Gamma(z_i, z_j)]\geq 0$.

\subsection{Separation, interpolation, and Carleson measures}
In order to state our theorem we need to develop some of the necessary
background on separation of points, interpolating sequences and Carleson
measures. We state our definitions in terms of reproducing kernels and
multiplier algebras. The case we are interested in is the RKHS
$H^2_\Gamma$ and its multiplier algebra $H^\infty_\Gamma$. In this
situation there is additional structure that we can exploit.

Let $X$ be a set and let $\{x_n\}$ be a sequence of points in $X$. Let
$H$ be a reproducing kernel Hilbert space of functions on $X$ with kernel $K$ and let
$M(H)$ be its multiplier algebra. We say that $\{x_n\}$ is an
\textit{interpolating sequence} for the algebra $M(H)$ if and only if
the restriction map $R:M(H)\to \ell^\infty$ given by $R(f) = \{f(x_n)\}$
is surjective.

Given a point $x\in X$ and a set $S\subseteq X$ we define the $M(H)$-distance from
$x$ to $S$ by $d_{M(H)}(x, S) = \sup\{\abs{f(x)}\,:\, f|_S = 0,
\norm{f}_{M(H)} \leq 1\}$.  The sequence  $\{x_n\}$ is called  
\textit{$M(H)$-separated} if and only if there exists a constant
$\delta >0$ such that $d_{M(H)}(x_n, Z_n)\geq \delta$ for all $n\geq
1$, where $Z_n = \{x_m\,:\, m\geq 1\}\setminus \{x_n\}$. If $\{x_n\}$ is
an interpolating sequence, then there exists a constant $C$ such that
for any sequence $w\in \ell^\infty$, $R(f) = w$ and
$\norm{f}_{M(H)}\leq C\norm{w}_\infty$. Applying this to the case
where $w = e_j$ we see that $d_{M(H)}(x_n , Z_n) \geq C^{-1}$. Therefore an interpolating sequence for $M(H)$ is $M(H)$-separated.

Carleson's theorem states that the converse is true for
$H^\infty(\mathbb{D})$, that is, every $H^\infty$-separated sequence
is an $H^\infty$-interpolating sequence. The modern approach to this problem
relies on the fact that the $H^\infty$ is the multiplier algebra of
the Hardy space, and the fact that the Szeg\"o kernel has the complete
Pick property.  We will use a similar approach based on Theorem~\ref{interpthm} and
bootstrap our results to the case of $H^\infty$.

There is a related notion of separation in terms of the reproducing
kernel of $H$.  In~\cite{AgMc2} it is shown that the function $\rho_H(x,y)
= \sqrt{1 - \frac{\abs{K(x,y)}^2}{K(x,x)K(y,y)}}$ is semi-metric on
the set $X$. A sequence of points is called \textit{$H$-separated} if
and only if $\inf_{i\not = j}\rho_H(x_i,x_j) > 0 $. A sequence is
\textit{weakly separated} if and only if there is a constant $\delta>0$ and functions
$f_{i,j}\in M(H)$ such that $\norm{f_{i,j}}\leq 1$ with $f_{i,j}(x_i)
= \delta$ and $f_{i,j}(x_j) = 0$. In general a weakly separated sequence is
$H$-separated, and the converse if true for the case of Riemann
surfaces.

\begin{lm}\label{weaksep}
  Let $\Gamma$ be a finitely generated discrete group of automorphisms
  and let $H^\infty_\Gamma$ be the corresponding fixed-point
  algebra. There exists a constant $C$ such that for any pair of
  points $z,w\in \mathbb{D}$, $\rho_H(z,w) \geq \delta/C$ if and only
  if there exists a function $f\in H^\infty_\Gamma$ such that $f(w) =
  \delta$, $f(z) = 0$ and $\norm{f}_\infty\leq 1$.
\end{lm}
\begin{proof}
  By the Interpolation Theorem~\ref{interpthm} there exists a constant $C$ and function $f\in
  H^\infty_\Gamma$ such that $f(w) = \delta$ and $f(z) = 0$ with
  $\norm{f}_\infty \leq 1$ if and only if the matrix
\[
\begin{bmatrix}
C^2K^{\Gamma}(z,z) & C^2K^{\Gamma}(z,w)\\
C^2K^{\Gamma}(w, z) & (C^2-\delta^2)K^{\Gamma}(w,w)
\end{bmatrix} \geq 0,
\]
where $C$ is a constant that does not depend on the points
$z,w$. Since the diagonal terms of the above matrix are non-negative,
the matrix positivity condition is equivalent to the determinant being
non-negative. Computing the determinant and rearranging we find that
$\rho_H(z,w) \geq \delta/C$.
\end{proof}

\begin{cor}
A sequence is $H^2_\Gamma$-separated if and only if
the sequence is weakly separated by $H^\infty_\Gamma$.
\end{cor}

A sequence is called a (universal) interpolating sequence for $H$ if
and only if the map $T:H\to \ell^2$ given by $T(f) =
\left\{\frac{f(x_n)}{K(x_n,x_n)^{1/2}}\right\}$ is surjective. 

In order to state our results we need the notion of a Carleson measure. A measure $\mu$ on a set $X$ is called a Carleson measure for the Hilbert space $H$ if and only if there exists a constant $C(\mu)$ such that 
\[\int_X \abs{f(x)}^2\,d\mu \leq C(\mu)\norm{f}_H^2.\] 
Given a sequence of points $\{x_n\}$ we can construct a measure on the set $X$ by setting $\mu = \sum_{n=1}^\infty K(x_n,x_n)^{-1}\delta_{x_n}$. When $f\in H$, we see that 
\begin{eqnarray*}
\int_X\abs{f(x)}^2\,d\mu & = & \sum_{n=1}^\infty \abs{f(x_n)}^2K(x_n,x_n)^{-1}\\
&  = & \sum_{n=1}^\infty \abs{\inp{f}{k_{x_n}/\norm{k_{x_n}}}}^2\\
& \leq & C(\mu)\norm{f}_H^2.
\end{eqnarray*} 
With the last inequality happening if the set of points $\{x_n\}$ that generates $\mu$ generates a Carleson measure.

It is helpful to restate the above in terms of sequences in Hilbert
space. To this end let us fix a sequence $x_n \in X$, let $k_{x_n}$ be
the corresponding reproducing kernel and let $g_n =\frac{
k_{x_n}}{\norm{k_{x_n}}}$. Note that $\{g_n\}$ is a unit norm sequence in
the Hilbert space $H$. The map $T:H\to \ell^2$ given by $T(f) =
\left\{\frac{f(x_n)}{K(x_n,x_n)^{1/2}}\right\} = \{\inp{f}{g_n}\}$. It is well known that
this map is bounded if and only if $\{g_n\}$ is a Bessel sequence, i.e.,
there is a constant $C$ such that $\sum_{n=1}^\infty
\abs{\inp{f}{g_n}}^2 \leq C\norm{f}^2$ for all $f\in H$. This in turn
is equivalent to the fact that the measure $\sum_{n=1}^\infty K(x_n, x_n)^{-1}\delta_{x_n}$ is a Carleson measure for $H$. In order to make the connection with Pick interpolation later on, we also point out that the sequence $\{g_n\}$ is Bessel if and only if the Gram matrix $G$ whose entries are given by 
$\inp{g_j}{g_i}$ is bounded, when viewed as an operator on $\ell^2$.

The sequence $\{x_n\}$ is interpolating for $H$ if and only if the sequence
$\{g_n\}$ is a \textit{Riesz basic sequence}, i.e., the sequence $\{g_n\}$ is
similar to a orthonormal set. In terms of the Gramian this means that
$G$ is both bounded and bounded below.

Before we proceed we need a preliminary lemma that relates the fact
that $\{g_n\}$ is a Riesz basic sequence to the matrix positivity
condition that appears in Theorem~\ref{interpthm}

\begin{lm}\label{pick:gram}
  Let $C>0$ be a constant. Let $\{g_n\}$ be a sequence of vectors in a Hilbert space. The Gramian $[\inp{g_j}{g_i}]$ is both bounded, and bounded
  below if and only if for all points $(w_n)_{n\geq 1}\in \mathrm{ball}(\ell^\infty)$ the matrix $[(C^2-w_i\overline{w_j})\inp{g_j}{g_i}]$ is a positive matrix.
\end{lm}
\begin{proof}
  Suppose that $[(C^2 - w_i\overline{w_j})\inp{g_j}{g_i}] \geq 0$. If
  $\alpha_n$ is a sequence in $\ell^2$, then we get
  \[C^2 \sum_{i,j=1}^\infty \alpha_j\overline{\alpha_i}\inp{g_j}{g_i}
  \geq \sum_{i,j=1}^\infty
  w_i\overline{w_j}\alpha_j\overline{\alpha_i}\inp{g_j}{g_i}.\] Choose
  $w_i = \exp(2\pi t_i \sqrt{-1})$. This gives,
\[C^2 \sum_{i,j=1}^\infty \alpha_j\overline{\alpha_i}\inp{g_j}{g_i}
\geq \sum_{i,j=1}^\infty
\exp(2\pi(t_i-t_j)\sqrt{-1})\alpha_j\overline{\alpha_i}\inp{g_j}{g_i}.\]
If we integrate both sides from $0$ to $2\pi$ with respect to each of
the variables $t_1,\ldots,t_n,\ldots$ then the above equation reduces
to
\[C^2\norm{\sum_{n=1}^\infty \alpha_n g_n}^2 \geq \sum_{n=1}^\infty \abs{\alpha_n}^2.\]

Now choose $\alpha_i = w_i\beta_i$. We have
\[C^2 \sum_{i,j=1}^\infty
\exp(2\pi(t_i-t_j)\sqrt{-1})\beta_j\overline{\beta_i} \inp{g_j}{g_i}
\geq \sum_{i,j=1}^\infty \beta_j\overline{\beta_i}\inp{g_j}{g_i}.\]
Integrating as before, we obtain
\[C^2 \sum_{n=1}^\infty \abs{\beta_n}^2 \geq \norm{\sum_{n=1}^\infty
  \alpha_ng_n}^2.\]

For the converse assume that $B \geq [\inp{g_j}{g_i}] \geq
B^{-1}$. Let $D_w$ be the matrix with diagonal entries
$\{w_n\}$. Since $D_w$ is a contraction we obtain the equation
\[BG \geq D_wGD_w^*.\] Since $G$ is invertible we have
$\norm{G^{-1/2}D_w G^{1/2}} \leq \norm{G^{-1/2}}\norm{G^{1/2}} \leq
B$.  Therefore,
$$
B^2I - (G^{-1/2}D_w G^{1/2})(G^{-1/2}D_w G^{1/2})^*
\geq 0.
$$ 
Which gives $B^2G \geq D_wGD_w^*$. This is just $[(B^2 - w_i\overline{w_j})\inp{g_j}{g_i}]\geq 0$.
\end{proof}

\subsection{Proof of Theorem~\ref{interpriemann}}

Our goal is to prove Theorem \ref{interpriemann} (stated again for ease):
\begin{thm}
Let $Z = (z_n)\subseteq \mathbb{D}$ be a sequence of points in $H^\infty_\Gamma$ such that no two points lie on the same orbit of $\Gamma$, where $\Gamma$ is the group of deck transformation associated to a finite Riemann surface. Let $Z_n = Z \setminus \{z_n\}$. The following are equivalent: 
  \begin{enumerate}
  \item\label{interp1} The sequence $\{z_n\}$ is interpolating for $H^\infty_\Gamma$; 
  \item\label{interp2} The sequence $\{z_n\}$ is interpolating for $H^2_\Gamma$;
  \item\label{cm} The sequence $\{z_n\}$ is $H^2_\Gamma$-separated and $\sum_{n=1}^\infty
    K^\Gamma(z_i,z_i)^{-1}\delta_{z_i}$ is a Carleson measure;
  \item\label{bb} The Gramian $G = \left[\frac{K^\Gamma(z_i,z_j)}{\sqrt{K^\Gamma(z_i,z_i)K^\Gamma(z_j,z_j)}}\right]$ is bounded below;
  \item\label{sep} There is a constant $\delta>0$ such that $\inf_{n\geq 1}
    d_{H^\infty_\Gamma}(z_n, Z_n)\geq \delta$.
  \end{enumerate}
\end{thm}
The strategy that we will follow is to prove that (1) and (2) are
equivalent and that (4) and (1) are equivalent. We will then show that (5) implies (1). Finally, we establish that (2) implies (3) and that
(3) implies (5).  We now prove the first of our claims that lead to
the proof of Theorem~\ref{interpriemann}.

We fix notation as follows: Let $\{z_n\}\subseteq \mathbb{D}$ be a
sequence of points and let $K^{\Gamma}_{z_n}$ be the reproducing
kernel for the space $H^2_\Gamma$ at the point $z_n$. The
corresponding normalized kernel function will be denoted $g_n$. The
Gram matrix will be denoted $G = [\inp{g_j}{g_i}]$.

\begin{prop}
  {\rm (1)} $\Leftrightarrow$ {\rm (2)} A sequence of points $\{z_n\}$ is interpolating for $H^\infty_\Gamma$
  if and only if it is interpolating for $H^2_\Gamma$.
\end{prop}
\begin{proof}
  Suppose that $\{z_n\}$ is a sequence of points in $\mathbb{D}$. Let
  $g_n$ be the normalized kernel function for $H^2_\Gamma$ at the
  point $z_n$.  The restriction map $R:H^\infty_\Gamma\to \ell^\infty$
  is surjective if and only if there is a constant $M$ such that for
  every sequence $w = (w_n) \in \mathrm{ball}(\ell^\infty)$ there is a
  function $f\in H^\infty_\Gamma$ of norm at most $M$ such that $R(f)
  = w$. By the Nevanlinna--Pick type Theorem~\ref{interpthm} this is equivalent
  to the matrix $[(C^2M^2 - w_i\overline{w _j})\inp{g_j}{g_i}]\geq 0$
  for all choices of $w$. Using Lemma~\ref{pick:gram} we see that this
  is equivalent to the Gramian $G$ being bounded and bounded below
  which in turn is equivalent to the sequence $\{g_n\}$ being a Riesz
  basic sequence. By our comments earlier the sequence $g_n$ is a
  Riesz basic sequence if and only if $T:H^2_\Gamma\to \ell^2$ given
  by $T(f) = \{\inp{f}{g_n}\}$ is bounded and surjective, i.e., $\{z_n\}$
  is interpolating for $H^2_\Gamma$.
\end{proof}

\begin{prop}{\rm (5) $\Rightarrow$ {\rm (1)}}
If there is a constant $\delta$ such that $d_{H^\infty_\Gamma}(z_n, Z_n) \geq \delta >0$, then $\{z_n\}$ is an interpolating sequence for $H^\infty_\Gamma$.
\end{prop}
\begin{proof}

Given a set $Z\subseteq \D$, let $\Gamma Z := \{\gamma(z)\,:\,
\gamma\in\Gamma, z\in Z\}$.  Suppose that $d_{H^\infty_\Gamma}(z_n ,
Z_n)\geq \delta >0$. Then, by definition, there exist functions $f_n\in
H^\infty_\Gamma$ such that $\norm{f_n}_\infty\leq 1$,
$\abs{f_n(z_n)}\geq \delta$ and $f_n|_{Z_n} = 0$. It follows that
$f_n|_{\Gamma Z_n} = 0$. The proof of~\cite{S}*{Theorem 6.3} shows that
the sequence $\Gamma Z$ is an interpolating sequence for $H^\infty$.


Now given a sequence $\{w_n\}\in\ell^\infty$, let $\tilde{w}_{n,
  \gamma} = w_n$ for $\gamma\in\Gamma$ and $n \geq 1$. Since $\Gamma
Z$ is interpolating for $H^\infty$, there exists a function
$\tilde{f}\in H^\infty$ such that $\tilde{f}(\gamma(z_n)) =
\tilde{w}_{n, \gamma} = w_n$.

Next we invoke a result of Earle and Marden~\cite{EM}*{Theorem page 274}. Their result shows that there is a polynomial $p$  such that the map 
\[(\Phi g)(z) = \frac{\sum_{\gamma\in \Gamma}
  p(\gamma(z))g(\gamma(z))\gamma'(z)^2}{\sum_{\gamma\in
    \Gamma}p(\gamma(z))\gamma'(z)^2}\] defines a bounded projection
from $H^\infty$ onto $H^\infty_\Gamma$. If $g\in H^\infty$ and $g(\gamma(\zeta)) = c$ for some constant $c$, then $(\Phi g)(\zeta) = c$. It follows that the function $f = \Phi \tilde{f}$ is in $H^\infty_\Gamma$ and that $f(z_n) = w_n$. Hence, $z_n$ is an interpolating sequence for $H^\infty_\Gamma$. 
\end{proof}

\begin{prop}
{\rm (4)} $\Leftrightarrow$ {\rm (1)}
The Gram matrix is bounded below if and only if the sequence is interpolating for $H^\infty_\Gamma$.
\end{prop}
\begin{proof}
If $G\geq C^{-2} >0$, then
  by Theorem~\ref{interpthm}, there exists a function
  $F\in R(H^\infty_\Gamma)$ such that $\norm{F}\leq C$ and
  $F(z_n) = e_n$. If we write $F =(f_1,\ldots,)$, then $f_m(z_n) =
  \delta_{m,n}$ with $\norm{f_m}\leq C$. Let $\phi_n = f_n^2$. Given a sequence $w= \{w_n\}\in
  \ell^\infty$ let $f = \sum_{n=1}^\infty w_n f_n^2$. We have,
  \begin{align*}
\abs{f(z)} &\leq \abs{\sum_{n=1}^\infty w_nf_n(z)^2} \leq \sum_{n=1}^\infty \abs{w_n}\abs{f_n(z)}^2\\
&\leq \left(\sup_{n\geq 1}\abs{w_n}\right)  \norm{F(z)}^2 \leq  \norm{w}_{\ell^\infty} \norm{F}^2 \leq \norm{w}_{\ell^\infty}C^2.
\end{align*}
 This proves that the sequence is
  interpolating for $H^\infty_\Gamma$.

  If the sequence $z_n$ is interpolating for $H^\infty_\Gamma$, then
  for any choice of sequence $(w_n)\in \ell^\infty$ such that
  $\abs{w_n}\leq 1$, there exists a function $f\in H^\infty_\Gamma$,
  with $\norm{f}_\infty\leq C$ such that $f(z_n) = w_n$. Hence, the
  matrix $[(C^2 - w_i\overline{w_j})\inp{g_i}{g_j}]\geq 0$ for all
  $(w_n) \in \mathrm{ball}(\ell^\infty)$. From Lemma~\ref{pick:gram} we see
  that the Gramian is bounded below.
\end{proof}

\begin{prop} {\rm (2)} $\Rightarrow$ {\rm (3)} If $\{z_n\}$ is an
  interpolating sequence for $H^2_\Gamma$, then $\{z_n\}$ is
  $H^2_\Gamma$-separated and $\sum_{n=1}^\infty
  K^\Gamma(z_n,z_n)^{-1}\delta_{z_n}$ is a Carleson measure.
\end{prop}

\begin{proof}
This result is true for any RKHS and the proof can be found in~\cite{Seip}.
\end{proof}

\begin{prop}
{\rm (3)} $\Rightarrow$ {\rm (5)}
If the sequence $\{z_n\}$ is $H^2_\Gamma$-separated and $\sum_{n=1}^\infty K^\Gamma(z_n, z_n)^{-1}\delta_{z_n}$ is a Carleson measure, then the sequence $\{z_n\}$ is $H^\infty_\Gamma$-separated.
\end{prop}
\begin{proof}
  Fix one of the indices $m$. Since the sequence is $H^2_\Gamma$
  separated, by Lemma~\ref{weaksep} there exist functions $f_n$ such that
  $\norm{f_n}\leq 1$ with $f_n(z_n) = 0$ and $f_n(z_m) \geq \rho_{H^2_\Gamma}(z_n,z_m)$.

  We now consider the sequence of products $\phi_n = f_1\cdots
  f_n$. This sequence has a weak-* limit in the unit ball of
  $H^\infty_\Gamma$. Denote this limit by $\phi$. 

  The claim is that $\phi(z_m) > \delta'$ and $\phi(z_n) = 0$, where
  $\delta'$ is a constant that does not depend on $j$.

    To see this we note that the infinite product that defines
    $\phi(z_m)$ converges to a non-zero value if and only if the
    series $\sum_{n\not = m}1-\abs{f_n(z_m)}^2 < + \infty$. 

    Using the Carleson condition we get that $\sum_{n=1}^\infty
  K^{\Gamma}(z_n,z_n)\abs{K^{\Gamma}_{z_m}(z_n)}^2 \leq
  C\norm{K^\Gamma_{z_m}}^2$, where the constant $C$ is independent of
    $m$. Rewriting this we get 
\[\sum_{n = 1}^\infty
    \dfrac{\abs{K^\Gamma(z_m,z_n)}^2}{K^\Gamma(z_n,z_n)K^\Gamma(z_m,z_m)}
    \leq C.
\]
Now we
    invoke the fact that $f_n(z_m) \geq \rho_{H^2_\Gamma}(z_n,z_m)$
    from which we get that the  sum $\sum_{n\not = m}1-\abs{f_n(z_m)}^2\leq C$.
\end{proof}

Combining all these Propositions then gives the Proof of Theorem \ref{interpriemann}.

\subsection{Applications to the Feichtinger conjecture}
In this section we make some observations that are relevant to the
Kadison-Singer problem. This has been a significant problem in
operator algebras for the past 50 years. We refrain from stating the
problem in its original form, and instead focus on an equivalent
statement: the Feichtinger conjecture.

The Feichtinger conjecture asks whether every bounded frame $\{f_n\}$
can be written as the union of finitely many Riesz basic
sequences. In~\cite{CCLV} it is shown the term bounded frame can be
replaced by bounded Bessel sequence. The term bounded here means that
$\inf_{n\geq 1}\norm{f_n} > 0$. Perhaps, bounded below is a better
term. In recent years there has been interest in this problem from the
perspective of function theory. The frames of interest are sequences
of normalized reproducing kernels. Given a kernel function $K$ on a
set $X$, the normalized reproducing kernel at $x$ is the function $g_x
=\frac{k_x}{K(x,x)^{1/2}}$. Given a sequence of points in $X$ we obtain a
sequence of unit norm vectors $g_{x_n}$.

\begin{thm}
  Let $\{z_n\}$ be a sequence of points in the unit disk. 
  The Bessel sequence $\{g_{z_n}\}$ of reproducing kernels for $H^2_\Gamma$ can be written as
  a union of finitely many Riesz basic sequences. The Feichtinger
  conjecture is true for such Bessel sequences.
\end{thm}
\begin{proof}
  Let $\{z_n\}$ be a sequence of points in the unit disk. Let $K^{\Gamma}$ be the
  reproducing kernel for the space $H^2_\Gamma$. The sequence
  $\{g_{z_n}\}$ is a sequence of unit norm vectors in $H^2_\Gamma$. The
  condition that the sequence $\{g_{z_n}\}$ be a Bessel sequence is
  equivalent to the Carleson condition on the points $\{z_n\}$. A result
  of~\cite{M}, a proof of which can be found in~\cite{AgMc2}
  and~\cite{Seip}, shows that a Bessel sequence can be written as a
  union of finitely many $H$-separated sequences. By
  Theorem~\ref{interpriemann} an $H^2_\Gamma$-separated Bessel
  sequence of normalized reproducing kernels is an interpolating
  sequence for $H^2_\Gamma$. From Theorem~\ref{interpriemann} we see that the
  Bessel sequence of reproducing kernels for $H^2_\Gamma$ can be
  written as a union of finitely many Riesz basic sequences.
\end{proof}

\section{Interpolation for Products of Kernels}
In this section we prove a generalization of a theorem of
Agler-McCarthy~\cite{AgMc} on interpolating sequences in several
variables. Our results depend on the Pick Interpolation theorem due to
Tomerlin~\cite{T}. Our goal is to give a proof of Theorem~\ref{main}.

We proceed in much the same way as in \cite{AgMc} by first defining
related conditions that will help us in studying the equivalences
between the various interpolation problems.  Condition $(a)$ from
Theorem \ref{main} is equivalent to the following: There exists a
constant $M$ and positive semi-definite infinite matrices $\Gamma^j$,
$j=1,\ldots, d$, such that
\begin{equation}
\label{Conditiona'}
\tag{$a'$}
M\delta_{ij}-1=\sum_{l=1}^d\Gamma_{ij}^l\frac{1}{k_l}(\lambda_i^l,\lambda_j^l)
\end{equation}

While Condition $(b)$ is equivalent to the following: There exists a
constant $N$ and positive semi-definite infinite matrices $\Delta^j$,
$j=1,\ldots, d$, such that
\begin{equation}
\label{Conditionb'}
\tag{$b'$}
1 - N\delta_{ij}=\sum_{l=1}^d\Delta_{ij}^l\frac{1}{k_l}(\lambda_i^l,\lambda_j^l)
\end{equation}

\subsection{\texorpdfstring{Proof that $(a)\Leftrightarrow (a')$ and $(b)\Leftrightarrow(b')$}{Equivalence of (a) and (a') and (b) and (b')}}

Note that a kernel $K$ defined on a subset $Y\subseteq \D^d$ can
always be extended to a weak kernel $\tilde{K}$ on $\D^d$ by setting
$\tilde{K}(z,w) = K(z,w)$ for $z,w\in Y$ and $\tilde{K}(z,w) = 0$
otherwise.

We need to prove that if $K$ is a kernel on $\Lambda$ such that $(MI -
J)\cdot K\geq 0$, then $MI - J$ is a sum of the above form. The proof
of this follows from a basic Hilbert space argument. 

If $A, B \in M_n$, then the Schur product of $A$ and $B$ is the matrix
$A\cdot B = [a_{i,j}b_{i,j}]$. It is a well-known fact that the Schur
product of two positive matrices is positive. Let $M_n^h$ denote the
set of $n\times n$ Hermitian matrices. The space $M_n^h$ is a real
Hilbert space in the inner product $\inp{A}{B} = \trace(AB) =
\inp{A\cdot B e}{e}$ where $e$ is the vector in $\R^n$ all of whose
entries are 1.

If $C$ is a wedge in a Hilbert space $H$, then the dual wedge $C'$ is
defined as the collection of all elements $h$ of $H$ such that
$\inp{h}{x}\geq 0$ for all $x\in C$. The following observation appears in~\cite{P}.

\begin{prop}\label{dual}
  Let $C\subseteq M_n^h$ be a set of matrices such that for every
  positive matrix $P$, and $X\in C$, $P\cdot X \in C$. Then,
  \[C' = \{ H \in M_n^h\,:\, H\cdot X \geq 0 \text{ for all } X\in
  C\}.\]
\end{prop}
\begin{proof}
  If $H\cdot X\geq 0$, then $\inp{H}{X} = \inp{H\cdot X e}{e} \geq
  0$. On the other hand, assume that $H\in C'$ and let $v\in \C^n$. If
  $X = [x_{i,j}]$, then the matrix $v^*Xv = [v_ix_{i,j}\overline{v_j}]
  = (vv^*)\cdot X\in C$, since $vv^*$ is positive. Since $H\in C'$ we
  have that $\inp{H\cdot X v}{v} = \inp{H}{v^*Xv}\geq 0$.
\end{proof}

With this proposition in hand we can prove our claim. We are assuming
that $K$ is kernel function on $\Lambda$ and that $(MI - J)\cdot K\geq
0$.

Let $R_l$ be the matrix
$\left(\frac{1}{k_l}(\lambda_i^l,\lambda_j^l)\right)_{i,j=1}^n$. Note that $R_l$ is
self-adjoint. Let $\mathcal{R}_l$ be the set of matrices of the form
$P\cdot R_l$ where $P\geq 0$. Note that this collection is a closed
wedge that satisfies the hypothesis of Proposition~\ref{dual}. If $K$
is an admissible kernel, then $K\cdot R_l \geq 0$ and so $K\cdot P
\cdot R_l \geq 0$ for all $P$. Hence, $K\in \mathcal{R}_l '$ for $l = 1,\ldots, d$. 

Let $\mathcal{K}$ be the collection of positive matrices $K$ such that
$K\cdot R_l\geq 0$ for all $l$. We have just shown that $\mathcal{K} = \mathcal{R}_1'\cap
\cdots \cap \mathcal{R}_d'$.

Let $K$ be an admissible kernel such that $(MI - J)\cdot K \geq
0$. Note that any positive matrix $P$ such that $(MI-J)\cdot P\geq 0$
can be extended to an admissible kernel. This means that the matrix
$MI - J$ is in $\mathcal{K}'$. 

If $C_1, C_2$ and $C$ are closed wedges, then $(C_1\cap C_2)' = C_1'+C_2'$ and
$C'' = C$. Applying this result we get $\mathcal{K}' = (\mathcal{R}_1'\cap
\cdots \cap \mathcal{R}_d')' = \mathcal{R}_1'' + \cdots +
\mathcal{R}_d'' = \mathcal{R}_1 + \cdots + \mathcal{R}_d$. Hence, there exists matrices $\Gamma_l \geq 0$ such that $MI - J =
\sum_{l=1}^d \Gamma_l \cdot R_l$.

\subsection{\texorpdfstring{Equivalence of Conditions $(a')$ and $(b')$
    to vector-valued interpolation problems}{Equivalence of (a') and
    (b') to vector-valued interpolation problems}}

We next show that conditions $(a')$ and $(b')$ are equivalent to
certain vector-valued interpolation problems.  The general idea is to
follow the proof in \cite{AgMc}, but to use related results by
Tomerlin \cite{T} that are directly applicable to our setting.

First, some notation.  Let $E$ and $E_*$ denote separable Hilbert spaces
and let $\mathcal{L}(E,E_*)$ denote the space of bounded linear
operators from $E$ to $E_*$.  Let $\{e_i\}$ denote the standard basis
in $\ell^2(\N)$.

Finally, let $S_{H(k)}(\D^d; \mathcal{L}(E,E_*))$ denote the set of
functions $M$ such that there exist functions $H_j$ on $\D^d$ and
auxiliary Hilbert spaces $E_j$ with values in $\mathcal{L}(E_j,E_*)$
such that

$$
I_{E_*}-M(z)M(w)^{*}=\sum_{j=1}^d\frac{1}{k_j}(z_j,w_j)H_j(z) H_j(w)^{*}.
$$

\begin{thm}[Tomerlin \cite{T}]
\label{interp}
  Let $z_1,\ldots,z_n$ be points in $\D^d$, let $x_1,\ldots,x_n\in
  \mathcal{L}(\mathcal{E}_*,\mathcal{H}_n)$, let $y_1,\ldots,y_n\in
  \mathcal{L}(\mathcal{E}, \mathcal{H}_n)$. Then there exists an
  element $W\in S_{H(k)}(\D^d; \mathcal{L}(\mathcal{E},
  \mathcal{E}_*))$ such that $x_i W(z_i) = y_i$ if and only if there
  exist block matrices $[\Gamma_{i,j}^l]$ such that $x_ix_j^* -
  y_iy_j^* = \sum_{l=1}^d \Gamma^l_{i,j}\frac{1}{k_l}(z_i, z_j)$
\end{thm}
With this notation and result, we can now state an alternate
equivalence between condition $(b')$.
\begin{lm}
\label{b'equivb''}
Let $\{\lambda_j\}$ be a sequence of points in $\D^d$.
The following are equivalent:
\begin{itemize}
\item[$(b')$]There exists a constant $N$ and positive semi-definite
  infinite matrices $\Delta^j$, $j=1,\ldots, d$, such that
\begin{equation*}
1 - N\delta_{ij}=\sum_{l=1}^d\Delta_{ij}^l\frac{1}{k_l}(\lambda_i^l,\lambda_j^l);
\end{equation*}
\item[$(b'')$] There exists a function $\Phi\in S_{H(k)}(\D^d;
  \mathcal{L}(\C,\ell^2(\N)))$ of norm at most $\sqrt{N}$ such that
$$
\Phi(\lambda_i)=e_i.
$$
\end{itemize}
\end{lm}

Similarly, we have the following lemma giving an equivalent condition
for $(a')$.
\begin{lm}
\label{a'equiva''}
Let $\{\lambda_j\}$ be a sequence of points in $\D^d$.  The following
are equivalent:
\begin{itemize}
\item[$(a')$]There exists a constant $M$ and positive semi-definite
  infinite matrices $\Gamma^j$, $j=1,\ldots, d$, such that
\begin{equation*}
M\delta_{ij}-1=\sum_{l=1}^d\Gamma_{ij}^l\frac{1}{k_l}(\lambda_i^l,\lambda_j^l);
\end{equation*}
\item[$(a'')$] There exists a function $\Psi\in S_{H(k)}(\D^d;
  \mathcal{L}(\ell^2(\N),\C))$ of norm at most $\sqrt{M}$ such that
$$
\Psi(\lambda_i)e_i=1.
$$
\end{itemize}
\end{lm}

\begin{proof}[Proof of Lemma \ref{b'equivb''}]
  Suppose that $(b')$ is true. Consider the interpolation problem with
  $x_i = \sqrt{N} \in \C$ and $y_i = e_i$ viewed as a map from
  $\ell^2(\N)$ to $\C$. Then $x_ix_j^* - y_iy_j^* = NJ - I$. From the
  interpolation Theorem \ref{interp} we see that there exits an element
  $\tilde{\Psi} \in S_{H(k)}(\D^d; \mathcal{L}(\ell^2(\N), \C))$ such
  that $\sqrt{N}\tilde{\Psi}(\lambda_i) = y_i = e_i$. Hence, $\Psi =
  \sqrt{N}\tilde{\Psi}$ has norm at most $\sqrt{N}$ and has the
  property that $\Psi(\lambda_i) = e_i$.

  The converse follows from the fact that a multiplier $\Psi$ of norm
  at most $\sqrt{N}$ has the property that $N -
  \frac{1}{N}\Psi(\lambda)\Psi(\mu)^* = \sum_{l=1}^d
  \frac{1}{k_l}(\lambda,\mu) \Gamma^l(\lambda, \mu)$ for some positive
  semidefinite functions $\Gamma^1,\ldots,\Gamma^d$. When restricted
  to the points $\lambda_i$ we see that $NJ-I$ is a sum of the
  appropriate form.
\end{proof}

The proof of Lemma \ref{a'equiva''} is similar to the above.  We have seen that condition $(a)$ is equivalent to the condition $(a')$ which is equivalent to $(a'')$. A similar equivalence is true for $(b)$, $(b')$ and $(b'')$.

Before we prove~Theorem~\ref{main} recall that strong separation of a sequence $\{\lambda_n\}$ by $S_{H(k)}$ if there is a constant $M>0$ and functions $f_n\in S_{H(k)}$ such that $\norm{f_n}_{S_{H(K)}}\leq M$, $f_n(\lambda_n) = 1$, and $f_n(\lambda_m) = 0$ for $n\not = m$.

\subsection*{Proof of Theorem~\ref{main}}

We are now ready to prove the second main theorem (stated again for ease on the reader)

\begin{thm}
  Let $\{\lambda_j\}$ be a sequence of points in $\D^d$.  The
  following are equivalent:
\begin{itemize}
\item[(i)] $\{\lambda_j\}$ is an interpolating sequence for $S_{H(k)}(\D^d)$;
\item[(ii)] The following two conditions hold
\begin{itemize}
\item[$(a)$] For all admissible kernels $k$, their normalized Gramians
  are uniformly bounded above,
$$
G^k\leq MI
$$
for some $M>0$,
\item[$(b)$] For all admissible kernels $k$, their normalized Gramians
  are uniformly bounded below,
$$
G^k\geq NI
$$
for some $N>0$;
\end{itemize}
\item[(iii)] The sequence $\{\lambda_j\}$ is strongly separated and
   condition $(a)$ alone holds;
 \item[(iv)] Condition $(b)$ alone holds.
\end{itemize}
\end{thm}

\begin{proof}

Let $\lambda_i$ be an
  interpolating sequence and suppose that $k$ is an admissible
  kernel. Let $k_j$ denote the normalized kernel function at the point
  $\lambda_j$. Let $w_i$ be a sequence of points such that $\abs{w_i}
  \leq 1$ for all $i$.

  Using the interpolation Theorem~\ref{interp} we see that there
  exists a function $f$ such that $f/\sqrt{M} \in S_{H(k)}$ and
  $f(\lambda_i) = w_i$ if and only if the matrix $(M -
  w_i\overline{w_j}) \cdot \inp{k_j}{k_i} \geq 0$ for all admissible
  kernels $k$. This statement is equivalent to the fact that
  $M\norm{\sum_{i=1}^\infty \alpha_i k_i}^2 \geq
  \norm{\sum_{i=1}^\infty \alpha_i w_i k_i}^2$ for all sequences
  $\{\alpha_i\}\in\ell^2$.

  It follows from the argument in~\cite{AgMc}*{Lemma 2.1} that both
  $(M\delta_{i,j} - J)\cdot K $ and $(J - M\delta_{i,j})\cdot K$ are
  positive.

It is also clear that (i) and (ii) are equivalent to the conditions
(iii) and (iv).

We now come to the equivalence of (iii) and (iv). The proof that (iii)
and (iv) are equivalent is essentially that given by
Agler-McCarthy~\cite{AgMc}.  If there exists $\Phi$ such that $\Phi/\sqrt{M} \in S_{H(k)}(\D^d,
\mathcal{L}(\C,\ell^2))$ such that $\Phi(\lambda_i)^*e_i = 1$, then
writing $\Phi = (\phi_1,\ldots,)^t$ we see that
$\overline{\phi_i(\lambda_i)} = 1$. Since we have assumed strong
separation, there exist $f_i$ and a constant $C$ such that
$f_i(\lambda_j) = \delta_{i,j}$ and $\norm{f_i}\leq C$. Therefore
$\Psi = (\phi_1 f_1,\ldots)$ has the property that $\norm{\Psi} \leq
C\sqrt{M}$ and $\Psi(\lambda_i) = e_i$.

Conversely if $\Psi = (\psi_1,\ldots)$ has the property that
$\Psi(\lambda_i) = e_i$ then $\psi_j(\lambda_i) =
\delta_{i,j}$. Therefore the functions $\psi_i$ strongly separate
$\lambda_i$. Setting $\Phi = \Psi^t$ we see that $\Phi(\lambda_i)^*e_i
= \overline{\phi_i(\lambda_i)} = 1$.
\end{proof}



\end{document}